\documentclass[12pt]{article}

\usepackage{amssymb}%
\usepackage{ytableau}
\usepackage{amsmath}%
\usepackage{marvosym}
\usepackage{cjhebrew}
\usepackage{amsthm}%
\usepackage{bbm}
\usepackage{xcolor}%

\allowdisplaybreaks%

\def\shin{{\text{\large\</s>}}}

{\theoremstyle{plain}%
 \newtheorem{theorem}{Theorem}

 \newtheorem{lemma}{Lemma}%
}
{\theoremstyle{remark}

}
{\theoremstyle{definition}
\newtheorem{definition}{Definition}
\newtheorem{example}{Example}
}

\usepackage{tikz-cd}
\usepackage{nccmath}
\usepackage{genyoungtabtikz}
\usepackage[enableskew]{youngtab}
\usepackage{mathrsfs}

\begin{document}

\begin{center}
{\Large{Quasi-immanants}}

\vspace{0.05in} 

\hspace{0.0in} {\textsc{John M. Campbell}} 

 \ 

\end{center}

\begin{abstract}
 For an integer partition $ \lambda$ of $n$ and an $n \times n$ matrix $A$, consider the expansion of the immanant 
 $\text{Imm}^{\lambda}(A)$ as a sum indexed by permutations $\sigma$ of order $n$, with coefficients given by the irreducible 
 characters $\chi^{\lambda}(\text{ctype}(\sigma))$ of the symmetric group $S_{n}$, for the cycle type $\text{ctype}(\sigma) \vdash n$ of 
 $\sigma$. Skandera et al.\ have introduced combinatorial interpretations of a generalization of immanants given by replacing the 
 coefficient $\chi^{\lambda}(\text{ctype}(\sigma))$ with preimages with respect to the Frobenius morphism of elements among the 
 distinguished bases of the algebra $\textsf{Sym}$ of symmetric functions. Since $ \textsf{Sym}$ is contained in the algebra $ 
 \textsf{QSym}$ of quasisymmetric functions, this leads us to further generalize immanants with the use of quasisymmetric functions. 
 Since bases of $ \textsf{QSym}$ are indexed by integer compositions, we make use of cycle compositions in place of cycle types to 
 define the family of \emph{quasi-immanants} introduced in this paper. This is achieved through the use of the quasisymmetric power 
 sum bases due to Ballantine et al., and we prove a combinatorial formula for the coefficients arising in an analogue, given by a special 
 case of quasi-immanants associated with quasisymmetric Schur functions, of second immanants. 
\end{abstract}


\vspace{0.1in}


\section{Introduction}
 Matrix functions such as determinants and permanents form a fundamentally important part of advanced linear algebra. Matrix 
 functions known as \emph{immanants} were introduced in 1934 \cite{LittlewoodRichardson1934} and generalize both determinants 
 and permanents in a way that may be seen as forming a bridge between linear algebra and algebraic combinatorics. This generalization 
 due to Littlewood and Richardson is defined with the use of the irreducible characters of symmetric groups, in such a way so that every 
 integer partition of $n$ gives rise to an immanant for a given $n \times n$ matrix. Immanants have important applications in many 
 different areas such as algebraic graph theory, and this motivates research endeavors based on the introduction of and application of 
 generalizations of immanants. Previously, we introduced a generalization of immanants using the characters of diagram algebras 
 containing symmetric group algebras \cite{Campbelltoappear}. In contrast, we consider, in our current paper, generalizing immanants 
 using an extension of symmetric \emph{functions}, as opposed to symmetric \emph{groups}. 

 An integer partition is a finite tuple of weakly decreasing positive integers. For an integer partition $\lambda$, if the sum of the entries 
 of $\lambda$ equals $n$, we write $\lambda \vdash n$, and we let the order $|\lambda|$ of $\lambda$ be equal to $n$. We let the 
 length $\ell(\lambda)$ of $ \lambda$ refer to the number of its entries, writing $\lambda = (\lambda_{1}, \lambda_{2}, \ldots, 
 \lambda_{\ell(\lambda)})$. For a permutation $ \sigma$ in the symmetric group $S_{n}$, we let $\text{ctype}(\sigma)$ denote the 
 partition of $n$ obtained by sorting the cycle lengths of $\sigma$. For $ \lambda \vdash n$, we also let $\chi^{\lambda}(\mu)$ denote 
 the value of the irreducible character $\chi^{\lambda}$ of $S_{n}$ evaluated at any permutation $\sigma \in S_{n}$ such that 
 $\text{ctype}(\sigma) = \mu$. For an $n \times n$ matrix $A = [a_{i, j}]_{1 \leq i, j \leq n}$ and for $\lambda \vdash n$, the 
 \emph{immanant} $\text{Imm}^{\lambda}$ may then be defined so that 
\begin{equation}\label{Immdefinition}
 \text{Imm}^{\lambda}(A) = \sum_{\sigma \in S_{n}} \chi^{\lambda}(\text{ctype}(\sigma)) \prod_{i=1}^{n} a_{i, \sigma_{i}}, 
\end{equation}
 again with reference to Littlewood and Richardson's paper \cite{LittlewoodRichardson1934}. Skandera et al.\ 
 \cite{ClearmanSheltonSkandera2011,Skandera2021} have explored combinatorial applications of extensions of \eqref{Immdefinition} 
 based on the transition matrices between the usual bases of the algebra of symmetric functions. We further extend such immanant 
 functions, with the use of quasisymmetric functions, as opposed to symmetric functions. 

 We let power sum generators be given by formal sums of the form $ p_{n} = \sum_{i \geq 1} x_i^{n} $ for indeterminates $x_{i}$ for 
 $i \geq 1$, referring to MacDonald's text on symmetric functions for details \cite{Macdonald1995}. We then write 
\begin{equation}\label{pproductrule}
 p_{\mu} = p_{\mu_{1}} p_{\mu_{2}} \cdots p_{\mu_{\ell(\mu)}} 
\end{equation}
 for an integer partition $\mu$. We then let expressions of the form $s_{\lambda}$ for $\lambda \vdash n$ be defined by taking formal 
 linear combinations so that 
\begin{equation}\label{Schurtopower}
 s_{\lambda} = \frac{1}{n!} \sum_{\sigma \in S_{n}} \chi^{\lambda}( \text{ctype}(\sigma) ) p_{\text{ctype}(\sigma)}. 
\end{equation}
 Expressions of the form $s_{\lambda}$ are referred to as \emph{Schur functions} and are of vital importance within algebraic combinatorics. 
 The work of Skandera et al.\ \cite{ClearmanSheltonSkandera2011,Skandera2021} concerns generalizations of \eqref{Immdefinition} 
 using power sum expansions for the usual bases of symmetric functions, and, in contrast, we generalize \eqref{Immdefinition} with 
 the use of the \emph{quasisymmetric} power sums introduced by Ballantine et al.\ \cite{BallantineDaughertyHicksMasonNiese2020}.

 For a positive integer $n$, adapting notation from the work of Skandera \cite{Skandera2021}, we define the Frobenius map $\text{Frob} 
 = \text{Frob}_{n}$ (cf.\ \cite[\S4.7]{Sagan2001}) so that it maps an element $\theta$ in a space of symmetric group traces to 
 the symmetric function 
\begin{equation}\label{thetamapsto}
 \theta \mapsto \frac{1}{n!} \sum_{\sigma \in S_{n}} \theta(\sigma) p_{\text{ctype}(\sigma)}, 
\end{equation}
 observing how this can be thought of as generalizing \eqref{Schurtopower}, in the sense that $\chi^{\lambda} \mapsto s_{\lambda}$. 
 The Frobenius map in \eqref{thetamapsto} provides both an isometry and an algebra isomorphism between the space of class 
 functions on symmetric groups and symmetric functions, referring to Sagan's text for details \cite[\S4.7]{Sagan2001}. Since immanants 
 can be defined in a natural way according to \eqref{Schurtopower}, this leads to combinatorial topics and problems on the 
 generalization of immanants using the Froebnius mapping in \eqref{thetamapsto}, and this was considered in an equivalent way in the 
 work of Skandera et al.\ \cite{ClearmanSheltonSkandera2011,Skandera2021}. Since the Frobenius map is of fundamental importance in 
 terms of how it forms connections between symmetric functions and the representation theory of the symmetric group, the foregoing 
 points lead us to generalize immanants using a generalization of symmetric functions. 

 Informally, since the algebra $\textsf{QSym}$ of quasisymmetric functions is such that its bases are indexed by integer compositions, 
 whereas the bases of $\textsf{Sym}$ are indexed by integer partitions, it can be thought of as being advantageous to employ a 
 \emph{cycle composition} analogue of partitions of the form $\text{ctype}(\sigma) \vdash n$ involved in the above definition of 
 $ \text{Imm}^{\lambda}$, in the following sense. Since an integer composition is not necessarily equal to the tuple obtained by 
 permuting its entries, cycle compositions, as defined in Section \ref{sectiondefinition}, can be thought of as ``containing more 
 information'' about permutations relative to partitions of the form $\text{ctype}(\sigma)$. 

\section{Background}\label{sectionPrelim}
 For power sum generators as defined above in the set $\mathbb{Q}[\![ x_{1}, x_{2}, \ldots ]\!]$ of formal power series over $\mathbb{Q}$ 
 with the given indeterminates, we proceed to define the algebra $\textsf{Sym}$ of symmetric functions by writing $\textsf{Sym} = 
 \mathbb{Q}[p_{1}, p_{2}, \ldots]$, i.e., so that $\textsf{Sym}$ is the commutative subalgebra of $\mathbb{Q}[\![ x_{1}, x_{2}, \ldots ]\!]$ 
 generated by $\{ p_{1}, p_{2}, \ldots \}$. By setting the degree of $p_{n}$ as $n$ for positive integers $n$, we thus have that $ 
 \textsf{Sym}$ may equivalently be defined as the free commutative $\mathbb{Q}$-algebra with one generator in each degree. In 
 view of the product rule in \eqref{pproductrule}, we find that bases of $\textsf{Sym}$ are indexed by the set $\mathcal{P}$ of all 
 integer partitions, writing $p_{()} = 1$, giving us the \emph{power sum basis} $\{ p_{\lambda} \}_{\lambda \in \mathcal{P}}$ of $ 
 \textsf{Sym}$. The transition matrices associated with the expansion rule in \eqref{Schurtopower} then allow us to define the 
 \emph{Schur basis} $\{ s_{\lambda} \}_{\lambda \in \mathcal{P}}$ of $\textsf{Sym}$. 

 An integer composition refers to a finite tuple $\alpha$ of positive integers, and we let $\ell(\alpha)$ denote the number of entries of $ 
 \alpha$ and write $\alpha = (\alpha_{1}, \alpha_{2}, \ldots, \alpha_{\ell(\alpha)})$. If the sum of the entries of $\alpha$ is equal to a 
 given integer $n$, we let this property be denoted by writing $\alpha \vDash n$, and we let the order $|\alpha|$ of $\alpha$ be equal 
 to $n$. Also, we let $\mathcal{C}$ denote the set of all integer compositions, including the empty composition $()$, writing $() 
 \vDash 0$. Integer compositions are of basic importance in the combinatorial and algebraic study of quasisymmetric functions, 
 which we define as below. 

 For a nonempty element $\alpha$ in $\mathcal{C}$, we write $M_{\alpha}$ to denote the element in $\mathbb{Q}[\![ x_{1}, x_{2}, \ldots ] 
 \!]$ such that $$ M_{\alpha} = \sum_{i_{1} < i_{2} < \cdots < i_{\ell(\alpha)}} x_{i_{1}}^{\alpha_{1}} x_{i_{2}}^{\alpha_{2}} \cdots 
 x_{i_{\ell(\alpha)}}^{\alpha_{\ell(\alpha)}}, $$ and we write $M_{()} = M_{0} = 1$. This leads us to define $\textsf{QSym}$ as the 
 subalgebra of $\mathbb{Q}[\![ x_{1}, x_{2}, \ldots ]\!]$ spanned by $\{ M_{\alpha} \}_{\alpha \in \mathcal{C}}$, which provides a 
 basis of $\textsf{QSym}$ that we refer to as the \emph{monomial basis}. For an integer composition $\alpha$, we let $\text{sort}(\alpha) 
 $ denote the integer partition obtained by sorting the entries of $\alpha$, i.e., in weakly decreasing order. For a nonempty integer 
 partition $\lambda$, we set 
\begin{equation}\label{membed}
 m_{\lambda} = \sum_{\substack{ \alpha \in \mathcal{C} \\ \text{sort}(\alpha) = \lambda }} M_{\alpha} 
\end{equation}
 and we write $m_{()} = m_{0} = 1$. We have that $m_{\lambda} \in \textsf{Sym}$ for all $\lambda \in \mathcal{P}$, and that $\{ 
 m_{\lambda} \}_{\lambda \in \mathcal{P}}$ is a basis of $\textsf{Sym}$, i.e., the \emph{monomial basis} of $\textsf{Sym}$. 

 Since a main purpose of this paper is to generalize the definition in \eqref{Immdefinition} using elements of $\textsf{QSym}$, this 
 leads us to consider how Skandera et al.\ \cite{ClearmanSheltonSkandera2011,Skandera2021} have previously applied generalizations of 
 \eqref{Immdefinition} using the standard bases of $\textsf{Sym}$, as in with the elementary immanants defined below. 

 The elementary symmetric generator $e_{n} \in \textsf{Sym}$ for $n \geq 0$ is such that $e_{0} = 0$ and such that $$ e_{n} = 
 \sum_{1 \leq i_{1} < i_{2} < \cdots < i_{n}} x_{i_{1}} x_{i_{2}} \cdots x_{i_{n}} $$ for $n > 0$. By then writing $e_{\lambda} = 
 e_{\lambda_{1}} e_{\lambda_{2}} \cdots e_{\lambda_{\ell(\lambda)}}$, the \emph{elementary basis} of $\textsf{Sym}$ is $\{ e_{\lambda} 
 \}_{\lambda \in \mathcal{P}}$. Following \cite{ClearmanSheltonSkandera2011,Skandera2021}, we may define coefficients of the form 
 $\epsilon^{\lambda}(\sigma)$, for $\sigma \in S_{n}$ and for $\lambda \vdash n$, according to the expansion such that 
\begin{equation}\label{etop}
 e_{\lambda} = \frac{1}{n!} \sum_{\sigma \in S_{n}} \epsilon^{\lambda}(\sigma) p_{\text{ctype}(\sigma)}. 
\end{equation}
 The \emph{elementary immanant} $\text{Imm}^{\epsilon^{\lambda}}(A)$
 may then, according to \cite{ClearmanSheltonSkandera2011}, be defined so that 
\begin{equation}\label{Immepsilon}
 \text{Imm}^{\epsilon^{\lambda}}(A) 
 = \sum_{\sigma \in S_{n}} \epsilon^{\lambda}(\sigma) 
 \prod_{i=1}^{n} a_{i, \sigma_{i}}. 
\end{equation}

 Again with reference to \cite{ClearmanSheltonSkandera2011,Skandera2021}, we let, again for $\lambda \vdash n$ and for $\sigma \in 
 S_{n}$, scalars of the form $\phi^{\lambda}(\sigma)$ be such that 
\begin{equation}\label{mtop}
 m_{\lambda} = \frac{1}{n!} \sum_{\sigma \in S_{n}} \phi^{\lambda}(\sigma) p_{\text{ctype}(\sigma)}, 
\end{equation}
 i.e., so that $\phi^{\lambda}$ is the preimage of $m_{\lambda}$ with respect to $\text{Frob}$. 
 Following \cite{ClearmanSheltonSkandera2011}, 
 we then define the \emph{monomial immanant} $\text{Imm}^{\phi^{\lambda}}(A)$ so that 
\begin{equation}\label{Immphi}
 \text{Imm}^{\phi^{\lambda}}(A) 
 = \sum_{\sigma \in S_{n}} \phi^{\lambda}(\sigma) 
 \prod_{i=1}^{n} a_{i, \sigma_{i}}. 
\end{equation}

 We write $$ h_{n} = \sum_{1 \leq i_{1} \leq i_{2} \leq \cdots \leq i_{n}} x_{i_{1}} x_{i_{2}} \cdots x_{i_{n}} $$ to denote the $n^{\text{th}}$ 
 complete homogeneous generator of $\textsf{Sym}$ for integers $n > 0$, writing $h_{()} = h_{0} = 1$ and writing $h_{\lambda} = 
 h_{\lambda_{1}} h_{\lambda_{2}} \cdots h_{\lambda_{\ell(\lambda)}}$. The scalars in the expansions among \eqref{Schurtopower}, 
 \eqref{etop}, and \eqref{mtop} are determined by the usual (Hall) inner product structure on $\textsf{Sym}$ such that 
\begin{equation}\label{defineHall}
 \langle h_{\lambda}, m_{\mu} \rangle = \delta_{\lambda, \mu} = \begin{cases} 
 1 & \text{if $\lambda = \mu$,} \\ 
 0 & \text{if $\lambda \neq \mu$.} 
 \end{cases} 
\end{equation}
 In the work of Skandera \cite{Skandera2021}, the classical immanant function in \eqref{Immdefinition} was implicitly generalized to matrix functions given 
 by replacing the coefficient $ \chi^{\lambda}(\text{ctype}(\sigma))$ in \eqref{Immdefinition} with the Hall inner product with $p_{\text{ctype}(\sigma)}$ 
 as an argument and a fixed element of $\textsf{Sym}$ as an argument. Since $\textsf{QSym}$ is not self-dual, as reviewed below, it is not clear how an 
 analogue of the bilinear function in \eqref{defineHall} could be applied to generalize immanants with the use of 
 quasisymmetric functions. 

 As a noncommutative analogue of the generating set such that
 $ \textsf{Sym} = \mathbb{Q}[h_{1}, h_{2}, \ldots]$, we set 
\begin{equation}\label{defineNSym}
 \textsf{NSym} := \mathbb{Q}\langle \text{{\bf h}}_{1}, \text{{\bf h}}_{2}, \ldots \rangle 
\end{equation}
 for indeterminates $\text{{\bf h}}_{n}$ that we endow with a degrees so that $\text{deg}(\text{{\bf h}}_{n}) = n$, i.e., so that the structure 
 in \eqref{defineNSym} may be understood as the free $\mathbb{Q}$-algebra with one generator in each degree. This algebra was 
 introduced by Gelfand et al.\ \cite{GelfandKrobLascouxLeclercRetakhThibon1995} and is referred to as the \emph{algebra of 
 noncommutative symmetric functions}. For a nonempty integer composition $\alpha$, we write $\text{{\bf h}}_{\alpha} = 
 \text{{\bf h}}_{\alpha_{1}} \text{{\bf h}}_{\alpha_{2}} \cdots \text{{\bf h}}_{\alpha_{\ell(\alpha)}}$, and we set $\text{{\bf h}}_{()} = \text{{\bf 
 h}}_{0} = 1$, so that the family $\{ \text{{\bf h}}_{\alpha} \}_{\alpha \in \mathcal{C}}$ is a basis of $\textsf{NSym}$, i.e., the 
 \emph{complete homogeneous basis} of $\textsf{NSym}$. 

 For a vector space $V$ over a field $\mathbbm{k}$, the \emph{dual} $V^{\ast}$ of $V$ consists of linear morphisms from $V$ to 
 $\mathbbm{k}$, giving rise to a bilinear pairing $\langle \cdot, \cdot \rangle\colon V \otimes V^{\ast} \to \mathbbm{k}$. The algebras 
 $\textsf{NSym}$ and $\textsf{QSym}$ are dual, with respect to the graded dual Hopf algebra structures on the specified algebras, 
 according to the pairing 
\begin{equation}\label{pairingNSym}
 \langle\cdot, \cdot \rangle\colon \textsf{NSym} \otimes \textsf{QSym} \to \mathbb{Q}
\end{equation}
 such that 
\begin{equation}\label{pairhM}
 \langle \text{{\bf h}}_{\alpha}, M_{\beta} \rangle = \delta_{\alpha, \beta}, 
\end{equation} 
 by direct analogy with \eqref{defineHall}. 
 The bilinear mapping in \eqref{pairingNSym} 
 is of key importance in relation to the {quasisymmetric power sums} \cite{BallantineDaughertyHicksMasonNiese2020} 
 required for our construction. 

\section{Quasi-immanants}\label{sectiondefinition}
 Determinants naturally arise in terms of how the primary bases of the algebra of symmetric functions relate to one another, as in the 
 classical Jacobi--Trudi rules, and this gives rise to problems and research investigations on immanants of Jacobi--Trudi matrices 
 \cite{GouldenJackson1992Algebra,StanleyStembridge1993,Stembridge1992}. Further and notable research contributions that concern 
 both immanants and symmetric functions \cite{GouldenJackson1992Proc,Haiman1993,Lesnevich2024,NagarSivasubramanian2021,RhoadesSkandera2006}
 motivate our work and lead to questions as to the use of quasisymmetric functions in place of 
 symmetric functions and in relation to immanants. 
 To define the quasi-immanant function introduced in this paper, 
 we first require the notion of a \emph{cycle composition}, as defined below. 

\begin{definition}
   For a permutation $\sigma$, the \emph{cycle composition} $\text{ccomp}(\sigma)$ of $\sigma$ is the integer composition obtained by 
  taking the cycle  decomposition of $\sigma$, and by then writing each cycle as a sequence of integers in increasing order and sorting the  
  resultant cycles in increasing  order, lexicographically, and by then taking the consecutive lengths of the cycles in the resultant list.  
\end{definition}

\begin{example}
 Inputting 
\begin{verbatim}
permu = Permutations(7).list()[777];
print(permu);
print(permu.to_cycles());
print([len(t) for t in permu.to_cycles()]);
\end{verbatim}
 into {\tt SageMath}, we obtain the output
\begin{verbatim}
[2, 1, 5, 4, 6, 7, 3]
[(1, 2), (3, 5, 6, 7), (4,)]
[2, 4, 1]
\end{verbatim}
 and this illustrates that 
 $ \text{ccomp} \left(\begin{smallmatrix} 
1 & 2 & 3 & 4 & 5 & 6 & 7 \\ 
 2 & 1 & 5 & 4 & 6 & 7 & 3
\end{smallmatrix}\right) = (2, 4, 1)$. 
\end{example}

 To define the quasi-immanant matrix functions introduced in this paper,  we are to require one of the families of quasisymmetric power 
 sums  introduced by Ballantine et al.\ \cite{BallantineDaughertyHicksMasonNiese2020}.  In this direction, we begin by defining the 
 \emph{descent set} of an integer composition $\alpha$ so that  $$ \text{Set}(\alpha) = \{ \alpha_{1},  \alpha_{1} + \alpha_{2}, \ldots, 
 \alpha_{1} + \alpha_{2} + \cdots + \alpha_{\ell(\alpha) - 1} \}. $$  We also write  $ \beta \preceq \alpha $ to denote the relation such 
 that $\text{Set}(\alpha) \subseteq \text{Set}(\beta)$.  For $\beta$ and $\alpha$ such that  $ \beta \preceq \alpha$, we write 
 $$ \alpha = (\beta_1 + \cdots + \beta_{i_{1}},  \beta_{i_{1} + 1} + \cdots + \beta_{i_{1} + i_{2}}, \ldots, \beta_{i_{1} + \cdots + i_{k-1} + 
 1} + \cdots + \beta_{i_{1} + \cdots + i_{k}} ). $$ We may then define $\beta^{(j)}$ as the composition given by the  consecutive entries 
 in $\beta$ given by the terms in the $j^{\text{th}}$ entry on the right-hand side of the above formula for $\alpha \succeq \beta$. 

 We then set  $$ \ell(\beta, \alpha) = \prod_{j=1}^{\ell(\alpha)} \ell\left( \beta^{(j)} \right). $$ Similarly, writing $\text{lp}(\beta) = 
 \beta_{\ell(\beta)}$, we set  $$ \text{lp}(\beta, \alpha) = \prod_{i=1}^{\ell(\alpha)} \text{lp}\left( \beta^{(i)} \right). $$ Also, for $\lambda 
 \vdash n$, letting $m_{i}$ denote the number of  parts of $\lambda$ of size $i$, and letting the maximal such size be $k$, we write 
\begin{equation}\label{zlambda}
 z_{\lambda} = 1^{m_{1}} m_{1}! 2^{m_{2}} m_{2}! \cdots k^{m_{k}} m_{k}!, 
\end{equation}
 so that \eqref{zlambda} is equal to the size of the stabilizer of a permutation that is of cycle type $\lambda$, 
 with respect to the action of the order-$n$ symmetric group on itself given by conjugation. 
 For an integer composition $\alpha$, following the work of 
 Ballantine et al.\ \cite{BallantineDaughertyHicksMasonNiese2020}, 
 we set $z_{\alpha} = z_{\text{sort}(\alpha)}$. 

 Type 1 and 2 \emph{quasisymmetric power sums} may be defined via the inverse relations such that 
\begin{equation}\label{Psidefinition}
 M_{\beta} = \sum_{\alpha \succeq \beta} (-1)^{\ell(\beta) - \ell(\alpha)} \frac{ \text{lp}(\beta, \alpha) }{z_{\alpha}} 
 \Psi_{\alpha} 
\end{equation}
 and such that 
\begin{equation}\label{Phidefinition}
 M_{\beta} = \sum_{\alpha \succeq \beta} (-1)^{\ell(\beta) - \ell(\alpha)} \frac{ \prod_{i} \alpha_{i} }{ z_{\alpha} \ell(\beta, \alpha) } 
 \Phi_{\alpha}. 
\end{equation}
 The inverse relations in \eqref{Psidefinition} and \eqref{Phidefinition}
 can be shown to provide a refinement of the power sum 
 symmetric functions in $\textsf{Sym}$, in that 
 $$ p_{\lambda} = \sum_{\text{sort}(\alpha) = \lambda} \Psi_{\alpha} 
 = \sum_{\text{sort}(\alpha) = \lambda} \Phi_{\alpha}. $$ 

\begin{definition}\label{defineQImm}
 For an $n \times n$ matrix $A = [a_{i, j}]_{1 \leq i, j \leq n}$ and for an element $Q$ in $\textsf{QSym}$,
 we define the \emph{quasi-immanant} 
 $\text{QImm}^{Q}_{\Psi}$ so that $$ \text{QImm}_{\Psi}^{Q}(A) = \sum_{\sigma \in S_{n}} 
 \left( \text{coefficient in $n! Q$ of} \begin{cases} 
 p_{\text{ctype}(\sigma)} & \text{if $Q \in \textsf{Sym}$} \\ 
 \Psi_{\text{ccomp}(\sigma)} & \text{otherwise} 
 \end{cases} \right) \prod_{i=1}^{n} a_{i, \sigma_{i}}, $$ 
 and similarly for $\text{QImm}^{Q}_{\Phi}$ (via the replacement of $\Psi$ with $\Phi$).
\end{definition}

 We find that Definition \ref{defineQImm} 
 provides a natural generalization of both classical immanants and the immanant-like functions considered by 
 Clearman et al.\ \cite{ClearmanSheltonSkandera2011}, in view of the special cases of 
 quasi-immanants given below. 
 Unless otherwise indicated, we henceforth let $A$ be an $n \times n$ matrix as in 
 Definition \ref{defineQImm}. 

\begin{theorem}\label{theoremSchur}
 For all $\lambda \vdash n$, the relation 
 $ \text{\emph{QImm}}_{\Psi}^{s_{\lambda}}(A) = \text{\emph{QImm}}_{\Phi}^{s_{\lambda}}(A) = \text{\emph{Imm}}^{\lambda}(A)$ holds. 
\end{theorem}

\begin{proof}
 This follows from the $s$-to-$p$ expansion formula in \eqref{Schurtopower}. 
\end{proof}

\begin{theorem} 
 For all $\lambda \vdash n$, the relation 
 $ \text{\emph{QImm}}_{\Psi}^{e_{\lambda}}(A) = \text{\emph{QImm}}_{\Phi}^{e_{\lambda}}(A) = \text{\emph{Imm}}^{\epsilon^{\lambda}}(A)$ holds. 
\end{theorem}

\begin{proof}
 This follows from the $e$-to-$p$ expansion formula in \eqref{etop}. 
\end{proof}

\begin{theorem} 
 For all $\lambda \vdash n$, the 
 relation $ \text{\emph{QImm}}_{\Psi}^{m_{\lambda}}(A) = 
 \text{\emph{QImm}}_{\Phi}^{m_{\lambda}}(A) = \text{\emph{Imm}}^{\phi^{\lambda}}(A)$ holds. 
\end{theorem}

\begin{proof}
 This follows from the $m$-to-$p$ expansion formula in \eqref{mtop}. 
\end{proof}

 For the relations for $\text{QImm}_{\ast}^{Q}$ covered above for $Q \in \textsf{Sym}$, 
 this leads us to consider the evaluation of $\text{QImm}_{\ast}^{Q}$ for non-symmetric and quasisymmetric functions $Q$. 
 Since $\chi^{\lambda}$ is the preimage of $s_{\lambda}$ with respect to the Frobenius morphism, 
 this leads us to let $Q$ be a quasisymetric analogue of $s_{\lambda}$, 
 in the hope that analogues of $\text{Frob}$ defined with the use of quasisymmetric functions 
 could be used to develop a better understanding of the matrix function in 
 \eqref{Immdefinition} and its relationships
 with \eqref{Immepsilon} and \eqref{Immphi}. 

\subsection{Quasi-Schur second immanants}
 Since the $\lambda = (1^{n})$ case of Theorem \ref{theoremSchur} is equivalent to 
\begin{equation}\label{208284818083803187PM8A}
 \text{QImm}_{\Psi}^{s_{(1^{n})}}(A) = \text{QImm}_{\Phi}^{s_{(1^{n})}}(A) = \text{det}(A). 
\end{equation}
 This leads us to consider quasisymmetric analogues of the relation in \eqref{208284818083803187PM8A}. 
 This, in turn, leads us to make use of quasisymmetric analogues of the Schur basis. 
 What are referred to as the \emph{canonical Schur-like bases of} 
 $\textsf{NSym}$ \cite{Campbell2016,MasonSearles2021,Searles2020Lothar,Searles2020Proc} 
 are the dual quasi-Schur basis $\{ \mathcal{S}_{\alpha}^{\ast} \}_{\alpha \in \mathcal{C}}$ 
 \cite{BessenrodtLuotovanWilligenburg2011}, 
 the immaculate basis $\{ \mathfrak{S}_{\alpha} \}_{\alpha \in \mathcal{C}}$ 
 \cite{BergBergeronSaliolaSerranoZabrocki2014}, 
 and the shin basis $\{ \shin_{\alpha} \}_{\alpha \in \mathcal{C}}$ \cite{CampbellFeldmanLightShuldinerXu2014}, 
 with each of these bases of $\textsf{NSym}$ providing a corresponding 
 and dual basis 
 of $\textsf{QSym}$, according to duality relations as in \eqref{pairhM}. 
 Out of these dual bases, for reasons described below, we will 
 mainly be concerned with with the quasi-Schur basis $\{ \mathcal{S}_{\alpha} \}_{\alpha \in \mathcal{C}}$
 of $\textsf{QSym}$. 

  The quasi-Schur function indexed by a composition $\alpha$  may be defined so that  
\begin{equation*}
 \mathcal{S}_{\alpha} = \sum_{\beta} K_{\alpha, \beta} M_{\beta} 
\end{equation*}
 holds, where $K_{\alpha, \beta}$ denotes the number of tableaux $T$ of shape 
 $\alpha$ and weight $\beta$ of the following form \cite{HaglundLuotoMasonvanWilligenburg2011}. 
 The cells of $T$ are filled with positive integers in 
 such a way so that 
 the rows of $T$ are weakly decreasing (read from left to right), 
 the first column is strictly increasing (when read from top to bottom), 
 and, letting $m$ be the largest part of $\alpha$, 
 by adding $0$-labeled cells to $T$ to form an $\ell(\alpha) \times m$
 rectangular tableau $\hat{T}$, then, for 
 $1 \leq i < j \leq \ell(\alpha)$ and $2 \leq k \leq m$, we have that 
 $$ \left( \hat{T}(j, k) \neq 0 \ \text{and} \ 
 \hat{T}(j, k) \geq \hat{T}(i, k) \right) \Longrightarrow \hat{T}(j, k) > \hat{T}(i, k-1). $$ 
 As for definitions for $\{ \mathfrak{S}_{\alpha}^{\ast} \}_{\alpha \in \mathcal{C}}$
 and $\{ \shin_{\alpha}^{\ast} \}_{\alpha \in \mathcal{C}}$, we refer to the 
 corresponding references that introduced the immaculate and shin bases 
 \cite{BergBergeronSaliolaSerranoZabrocki2014,CampbellFeldmanLightShuldinerXu2014}. 
 Little is known about the dual of the shin basis, and no explicit formulas are known 
 for expanding this dual basis into the usual bases of $\textsf{QSym}$, 
 so we omit consideration of quasisymmetric dual shin functions. 

 One might think that the determinant identities in \eqref{208284818083803187PM8A} 
 would suggest that expressions as in 
\begin{equation}\label{20250717076547470AM1A}
 \text{QImm}_{\ast}^{\mathfrak{S}^{\ast}_{(1^{n})}} 
 \ \ \ \text{and} \ \ \ \text{QImm}_{\ast}^{\mathcal{S}_{(1^{n})}} 
\end{equation} 
 would provide quasisymmetric analogues of 
 the determinant function according to 
 Definition \ref{defineQImm}, but 
 one may verify that $\mathfrak{S}_{\left( 1^{n} \right)}^{\ast} = \mathcal{S}_{\left( 1^{n} \right)}^{\ast} 
 = s_{\left( 1^{n} \right)}$, 
 i.e., so that all of the matrix functions in \eqref{20250717076547470AM1A} 
 reduce to the determinant. 
 This leads us to consider what is regarded as the ``next'' immanant 
 after the determinant, i.e., what is referred to as the second immanant 
 \cite{Grone1985,GroneMerris1984,GroneMerris1987Fischer,GroneMerris1987Hadamard,Merris1987Oppenheim,Merris1986,WuYuFengGao2024,WuYuGao2023}. 

 The \emph{second immanant} refers to the immanant function given by 
 setting $\lambda = (2, 1^{n-2})$ in the family of matrix functions in \eqref{Immdefinition}. 
 It can be shown that $\mathfrak{S}_{(2, 1^{n})}^{\ast} \in \textsf{Sym}$, 
 which leads us to show, as below, that $\mathcal{S}_{(2, 1^{n})} \not\in \textsf{{Sym}}$. 
 The following result gives us that 
 $\text{QImm}^{\mathcal{S}_{(2, 1^{n})}}_{\Psi}$ and $\text{QImm}^{\mathcal{S}_{(2, 1^{n})}}_{\Phi}$ do 
 not reduce to immanants or the generalization of immanants 
 relying on symmetric functions in the sense of the work of Skandera \cite{Skandera2021}. 

\begin{lemma}\label{notinSym}
 For positive integers $n$, we have that $\mathcal{S}_{(2, 1^{n})} \not\in \textsf{\emph{Sym}}$. 
\end{lemma}

\begin{proof}
 For compositions $\alpha, \beta \vDash n$, we write 
 $\alpha \blacktriangleright \beta$ 
 if $\text{sort}(\alpha) $ is lexicographically strictly greater than 
 $\text{sort}(\beta)$, or if $\text{sort}(\alpha) = \text{sort}(\beta)$ 
 and $\alpha$ is lexicographically strictly greater than $\beta$. 
 For the matrix with rows and columns indexed by compositions $\alpha \vDash n$ 
 and ordered by $\blacktriangleright$ 
 so that the $(\alpha, \beta)$-entry is the coefficient of $M_{\beta}$ in $\mathcal{S}_{\alpha}$, 
 this transition matrix is upper unitriangular, as shown by 
 Haglund et al.\ \cite[Proposition 6.7]{HaglundLuotoMasonvanWilligenburg2011}. 
 This gives us that $\mathcal{S}_{(2, 1^{n})}$ 
 equals the sum of $M_{(2, 1^{n})}$ and a scalar multiple of 
 $M_{(1^{n})}$, so that \eqref{membed} 
 gives us the desired result. 
\end{proof}

 This motivates the problem of evaluating $\text{QImm}^{\mathcal{S}_{(2, 1^{n})}}_{\ast}$. 
 Moreover, if we consider the special case of \eqref{Immdefinition} 
 for second immanants, 
 this leads us to the combinatorial interpretation for the resultant characters
 and to consider what would be appropriate as an analogue of this, with the use of quasisymmetric Schur functions. 
 For the $\lambda = (2, 1^{n-2})$ case of 
 \eqref{Immdefinition}, the character in 
 \eqref{Immdefinition} is such that $\chi_2(\sigma) = \text{sign}(\sigma)(f(\sigma) - 1)$, 
 writing $f(\sigma)$ in place of the number of gixed points of an order-$n$ permutation, with 
\begin{equation}\label{d2notation}
 d_{2}(A) = \sum_{\sigma \in S_{n}} \chi_{2}(\sigma) \prod_{i = 1}^{n} a_{i, \sigma(i)} 
\end{equation}
 according to the usual notation for second immanants, 
 writing $A = [a_{i, j}]_{n \times n}$. 
 As suggested above, the combinatorial interpretation of the character coefficients in 
 \eqref{d2notation} leads us to consider 
 how similar combinatorial properties could be determined for our quasisymmetric analogues of second immanants
 given by matrix functions of the form 
 $\text{QImm}^{\mathcal{S}_{(2, 1^{n-2})}}_{\ast}$. 
 We proceed to consider the 
 case for $\text{QImm}^{\mathcal{S}_{(2, 1^{n-2})}}_{\Psi}$, 
 by using the combinatorial formula in \eqref{Psidefinition}. 

\begin{theorem}\label{maintheorem}
 For a natural number $n \geq 3$ and for an $n \times n$ matrix $A = [a_{i, j}]_{1 \leq i, j \leq n}$, the quasi-immanant 
 $\text{\emph{QImm}}^{\mathcal{S}_{(2, 1^{n-2})}}_{\Psi}(A)$ 
 satisfies $$ \text{\emph{QImm}}^{\mathcal{S}_{(2, 1^{n-2})}}_{\Psi}(A) = 
 \sum_{\sigma \in S_{n}} 
 c_{\sigma} \prod_{i=1}^{n} a_{i, \sigma_{i}}, $$ 
 where the coefficient $ c_{\sigma} $ satisfies the following: 
\begin{enumerate}

\item If $\text{\emph{ccomp}}(\sigma)_{1} = 1$, then 
 $ c_{\sigma} $ equals $(-1)^{n - \ell(\text{\emph{ccomp}}(\sigma))}$ times the number of $n$-permutations
 of same cycle type as $\sigma$; 

\item If $\text{\emph{ccomp}}(\sigma)_{1} = 2$, then 
 $ c_{\sigma} $ equals $(-1)^{n - 1 - \ell(\text{\emph{ccomp}}(\sigma))}$ times the number of $n$-permutations
 of the same cycle type as $\sigma$; and 

\item If $\text{\emph{ccomp}}(\sigma)_{1} > 2$, then 
 $ c_{\sigma} = 0$. 

\end{enumerate}

\end{theorem}

\begin{proof}
 From \eqref{Psidefinition}, we find that 
\begin{equation}\label{factorialhook}
 n! M_{(2, 1^{n-2})} = \sum_{ {\substack{ \alpha \vDash n \\
 \alpha_{1} \geq 2 }} } (-1)^{ n - 1 - \ell(\alpha)} 
 \frac{n!}{z_{\alpha}} \text{lp}((2, 1^{n-2}), \alpha) \Psi_{\alpha} 
\end{equation}
 and that 
\begin{equation}\label{factorialstraight}
 n! M_{(1^n)} = \sum_{\alpha \vDash n } (-1)^{n - \ell(\alpha)} 
 \frac{n!}{z_{\alpha}} \text{lp}((1^n), \alpha) \Psi_{\alpha}. 
\end{equation}
 We see that \eqref{factorialstraight} reduces so that 
\begin{equation}\label{straightreduce}
 n! M_{(1^n)} = \sum_{\alpha \vDash n } (-1)^{n - \ell(\alpha)} 
 \frac{n!}{z_{\alpha}} \Psi_{\alpha} 
\end{equation}
 since, by taking $\alpha \succeq (1^{n})$, 
 by forming a given entry of $\alpha$ with consecutive parts of $(1^{n})$, 
 the last part will always be $1$. 
 Similarly, we may rewrite \eqref{factorialhook} 
 so that 
\begin{equation}\label{207727570712573427PM1A}
 n! M_{(2, 1^{n-2})} = 2 \sum_{ {\substack{ \alpha \vDash n \\ 
 \alpha_{1} = 2 }} } (-1)^{ n - 1 - \ell(\alpha)} 
 \frac{n!}{z_{\alpha}} \Psi_{\alpha} + \sum_{ {\substack{ \alpha \vDash n \\ 
 \alpha_{1} > 2 }} } (-1)^{ n - 1 - \ell(\alpha)} 
 \frac{n!}{z_{\alpha}} \Psi_{\alpha}. 
\end{equation}
 The $\mathcal{S}$-to-$M$ expansion formula gives us that $n! \mathcal{S}_{(2, 1^{n-2})} = n! M_{(2, 1^{n-2})} + n! M_{(1^{n})}$, and 
 Lemma \ref{notinSym} gives us that the coefficient of $\prod_{i=1}^{n} a_{i, \sigma_{i}}$ in the quasi-immanant 
 $\text{QImm}_{\psi}^{\mathcal{S}_{(2, 1^{n-2})}}(A)$ is equal to the coefficient in $n! \mathcal{S}_{(2, 1^{n-2})}$ 
 of $\psi_{\text{ccomp}(\sigma)}$. 
 By simplifying the sum of \eqref{straightreduce} and 
 \eqref{207727570712573427PM1A}, the desired result follows from 
 the property such that 
 $\frac{n!}{z_{\alpha}}$ equals 
 the number of permutations of order $n$ and of cycle type $\text{sort}(\alpha)$. 
\end{proof}

\begin{example}\label{exinequivalent}
 We find that the analogue of second immanants given in 
 Theorem \ref{maintheorem} is not equivalent to second immanants. 
 For example, consider the $3 \times 3$ case, with 
\begin{multline*}
 \text{QImm}_{\Psi}^{\mathcal{S}_{(2, 1)}} \left(
\begin{array}{ccc}
 a_{1, 1} & a_{1, 2} & a_{1, 3} \\
 a_{2, 1} & a_{2, 2} & a_{2, 3} \\
 a_{3, 1} & a_{3, 2} & a_{3, 3} \\
\end{array}
\right) = \\ 
 a_{1, 1} a_{2, 2} a_{3, 3} - 3 a_{1, 1} a_{2, 3} a_{3, 2}+3 a_{1, 2} a_{2, 1} a_{3, 3}+3 a_{1, 3} a_{2, 2} a_{3, 1}. 
\end{multline*}
 Observe that 
 $ \text{ccomp} \left(\begin{smallmatrix} 
1 & 2 & 3 \\ 
 1 & 3 & 2 
\end{smallmatrix}\right) = (1, 2)$, and that 
 the coefficient of $ a_{1, 1} a_{2, 3} a_{3, 2}$ is equal to 
 $(-1)^{3 - 2}$ times the number of 
 $3$-permutations of cycle type $(2, 1)$, i.e., 
 the number of permutations among 
 $ \text{ccomp} \left(\begin{smallmatrix} 
 1 & 2 & 3 \\ 
 1 & 3 & 2 
 \end{smallmatrix}\right)$ 
 and 
 $ \text{ccomp} \left(\begin{smallmatrix} 
 1 & 2 & 3 \\ 
 3 & 2 & 1 
 \end{smallmatrix}\right)$ 
 and 
 $ \text{ccomp} \left(\begin{smallmatrix} 
 1 & 2 & 3 \\ 
 2 & 1 & 3 
 \end{smallmatrix}\right)$. 
 For the second immanant of a $3\times 3$ matrix, we obtain 
 $$ d_{2} \left(
\begin{array}{ccc}
 a_{1, 1} & a_{1, 2} & a_{1, 3} \\
 a_{2, 1} & a_{2, 2} & a_{2, 3} \\
 a_{3, 1} & a_{3, 2} & a_{3, 3} \\
\end{array}
\right) 
 = 2 a_{1, 1} a_{2, 2} a_{3, 3}-a_{1, 2} a_{2, 3} a_{3, 1}-a_{1, 3} a_{2, 1} a_{3, 2}. $$
\end{example}

 The inequivalence illustrated in Example \ref{exinequivalent} 
 leads us to consider how the combinatorial rule in Theorem \ref{maintheorem}
 could be used to obtain identities for quasi-immanants of infinite families of matrices
 and how such identities relate to second immanants. 

\begin{example}
    Immanants of Toeplitz matrices are often studied in relation    to recursive properties of such immanants. This leads us to     consider   
  quasi-immanants of Toeplitz matrices, in relation to second immanants.     For example, we find that quasi-immanants of the form    
   $ \text{QImm}_{\Psi}^{\mathcal{S}_{(2, 1)}} $ provide a natural companion to     the second immanant relation such that   
$$ d_{2} \left(
\begin{array}{ccccccc}
 0 & 1 & 0 & \cdots & 0 & 0 & 0 \\
 1 & 0 & 1 & \ddots & 0 & 0 & 0 \\
 0 & 1 & 0 & \ddots & 0 & 0 & 0 \\
 \vdots & \ddots& \ddots & \ddots & \ddots & \ddots & \vdots \\
 0 & 0 & 0 & \ddots & 0 & 1 & 0 \\
 0 & 0 & 0 & \ddots & 1 & 0 & 1 \\
 0 & 0 & 0 & \cdots & 0 & 1 & 0 \\
\end{array}
\right)_{n \times n} = \begin{cases} 
 (-1)^{\frac{n}{2}+1} & \text{if $n$ is even,} \\ 
 0 & \text{otherwise,}
 \end{cases} $$
 with 
\begin{multline*}
 \text{QImm}_{\Psi}^{\mathcal{S}_{(2, 1^{n-2})}} \left(
\begin{array}{ccccccc}
 0 & 1 & 0 & \cdots & 0 & 0 & 0 \\
 1 & 0 & 1 & \ddots & 0 & 0 & 0 \\
 0 & 1 & 0 & \ddots & 0 & 0 & 0 \\
 \vdots & \ddots& \ddots & \ddots & \ddots & \ddots & \vdots \\
 0 & 0 & 0 & \ddots & 0 & 1 & 0 \\
 0 & 0 & 0 & \ddots & 1 & 0 & 1 \\
 0 & 0 & 0 & \cdots & 0 & 1 & 0 \\
\end{array}
\right)_{n \times n} = \\ 
 \begin{cases} 
 -\left(-\frac{1}{2} \right)^{\frac{n}{2}} \frac{n!}{ \left( \frac{n}{2} \right)! } & 
 \text{if $n$ is even,} \\ 
 0 & \text{otherwise.}
 \end{cases} 
\end{multline*}
\end{example}

  With regard to the combinatorial rule in   Theorem \ref{maintheorem},   similar results can be obtained for  
   $ \text{QImm}_{\Phi}^{\mathcal{S}_{(2, 1^{n-2})}} $  and for quasi-immanants more generally, and we encourage explorations of this.  

\subsection*{Acknowledgements}
 The author is thankful to acknowledge support from a Killam Postdoctoral Fellowship from the Killam Trusts
 and thanks Karl Dilcher and Lin Jiu for a useful discussion.

 \

{\textsc{John M. Campbell}} 

\vspace{0.1in}

 Department of Mathematics and Statistics

 Dalhousie University

 Halifax, NS, B3H 4R2, Canada

\vspace{0.1in}

 {\tt jh241966@dal.ca}

\end{document}